\documentclass[a4paper]{easychair}

\usepackage{amsmath,amssymb,amscd}
\usepackage{graphicx}
\usepackage{tikz}
\usepackage[latin1]{inputenc}
\usepackage{color}
\usepackage{multicol}
\usepackage{multirow}
\definecolor{red}{rgb}{1.0,0,0}
\newcommand{\Go}[1]{{G\"{o}del} }

\renewcommand{\to}{\mathop{\rightarrow}}

\usepackage{amsmath,amssymb,amscd}
\newtheorem{theorem}{Theorem}

\newtheorem{lemma}{Lemma}

\newtheorem{definition}{Definition}
\newtheorem{example}{Example}
\newtheorem{remark}{Remark}
\begin{document}

\date{ }
\title{Epistemic BL-Algebras}
\author{Manuela Busaniche\inst{1} \and Pen\'elope Cordero\inst{2} \and Ricardo Oscar Rodriguez \inst{3}}

\authorrunning{Busaniche, Cordero, Rodriguez} 
\titlerunning{Epistemic MV-algebras}
\institute{IMAL, CONICET-UNL. FIQ,UNL.
              {Santa Fe, Argentina.}\\
              \email{mbusaniche@santafe-conicet.gov.ar}           %
           \and
              IMAL, CONICET-UNL.
              {Santa Fe, Argentina.}\\
              \email{pcordero@santafe-conicet.gov.ar}           \\
           \and
           {UBA. FCEyN. Departamento de Computaci\'on.} \\
           {CONICET-UBA. Inst. de Invest. en Cs. de la Computaci\'on.}\\
           Buenos Aires, Argentina.\\
           \email{ricardo@dc.uba.ar}
}

\maketitle

\begin{abstract}
Fuzzy Epistemic Logic is an important formalism for approximate reasoning. It extends the well known basic propositional logic BL, introduced by H\'{a}jek, by offering the ability to reason about possibility and necessity of fuzzy propositions. We consider an algebraic approach to study this logic, introducing Epistemic BL-algebras. These algebras turn to be a generalization of both, Pseudomonadic Algebras introduced by \cite{Bez2002} and serial, euclidean and transitive Bi-modal G\"odel Algebras proposed by \cite{CaiRod2015}. We present the connection between this class of algebras and fuzzy possibilistic frames, as a first step to solve an open problem proposed by H\'{a}jek \cite[chap. ~8]{HajekBook98}.

\end{abstract}

\section{Introduction}
\label{intro}
 For many years classical logic has provided a formal basis to study human reasoning. However, human practical reasoning demands more than what traditional classical logic can offer. For instance, classically, the truth of a statement $q$ with respect to a state of knowledge {\cal K} is determined if every model of {\cal K} is also model of $q$. But nothing can be said about its truth value if only
the {\it most possible} models of {\cal K} are also models of $q$. The scene becomes more complicated if it is necessary to accept that the statement $q$ can also take an intermediate truth-value different from true and false. When we need to deal simultaneously with both fuzziness and modalities, a fuzzy version of epistemic logic should be an useful tool.

In the present paper, we want to characterize a fuzzy version of the classical epistemic logic $KD45$.
The usual semantics of epistemic logic is a Kripke-style semantics. This is why in \cite{HajekBook98} a Kripke semantics for a system of fuzzy epistemic logic is proposed. Unfortunately, it is not immediate to find an axiomatization of the corresponding logic starting from this semantics because the $K$ axiom is not valid. We attack the problem in a novel way, by proposing a possible algebraic semantic, which is obtained by extending BL-algebras (the algebraic models of basic logic) by two operators that model necessity and possibility. It turns out that the only minimal logics axiomatized in the literature are the ones where the base many-valued logic is the one corresponding to a finite Heyting algebra \cite{Fitting1,Fitting2}; the standard (infinite) G\"odel algebra \cite{CaiRod2015} or a finite residuated algebra \cite{BoEsGoRo11} (in particular finite {\L}ukasiewicz linearly ordered algebras).

To achieve our aim, we introduce a generalization of Monadic BL-algebras proposed by \cite{CCDR2017} which we call Epistemic BL-algebras (EBL-algebras). This generalization resembles what is done with monadic Boolean algebras and Pseudomonadic algebras in \cite{Bez2002}. In fact, we prove that the class of EBL-algebras whose BL-reduct is Boolean coincides with the class of Pseudomonadic algebras, indicating that Boolean epistemic BL-algebras are the algebraic counterpart of classical epistemic logic $KD45.$ This situation also happens for G\"{o}del EBL-algebras, since we prove that they are exactly serial, euclidean and transitive Bi-modal Algebras proposed by \cite{CaiRod2015}.
  We also introduce  a special class of epistemic BL-algebras, which we call c-EBL-algebras, that will become important to establish a connection with fuzzy possibilistic frames. We prove that a c-EBL-algebra ${\bf A}$ is completely determined by a pair formed by a subalgebra of ${\bf A}$ and an element of ${\bf A}.$  With these results, we recall the notion of possibilistic BL-frame of \cite{HajekBook98}, since its associated logic is the one that we want to characterize. We prove that each possibilistic BL-frame determines a unique c-EBL-algebra and we give necessary conditions for a BL-algebra of functions to be the reduct of a c-EBL-algebra corresponding to a possibilistic frame.  Finally, we conclude the paper describing some conclusions and future challenges.

Throughout this paper we assume that the reader is acquainted with basic notions concerning BL-algebras.

\section{Preliminaries}
\label{sec:0}

A \textit{residuated lattice} (\textit{commutative} and \textit{integral}) \cite{GJKO} is an algebra $${\bf A}=\langle A,\wedge, \vee, \ast,\to ,1 \rangle$$
of type $(2,2,2,2,0)$ such that $\langle A, \ast , 1\rangle,$ is a commutative monoid,
${\bf L(A)}:=\langle A, \wedge, \vee, 1\rangle $ is an upper bounded lattice and the following residuation condition holds:
\begin{equation}\label{eq:residuation}a\ast b\le c \ \ \ \mbox{ iff } \ \ \ a\le b \to c,\end{equation} where $\le$ is the order given by the lattice structure.
A \textit{bounded residuated lattice} ${\bf A}=\langle A,\wedge, \vee, \ast,\to ,0 ,1\rangle$
is an algebra that satisfies that the reduct $\langle A,\wedge, \vee, \ast,\to ,1 \rangle$ is a residuated lattice and the new constant $0$ is the lower bound of ${\bf L(A)}$. A \textit{BL-algebra} is a prelinear and divisible bounded residuated lattice, that is, a bounded residuated lattice that satisfies the equations
\begin{equation}\label{eq:prelinearity}
	a\wedge b=a\ast (a\to b),
\end{equation}
\begin{equation}\label{eq:divisibility}
	(a\to b)\vee (b\to a)=1.
\end{equation}
For details about BL-algebras \cite{HajekBook98, BuMon}. Some immediate consequences of the definition, that will be frequently used are:
\begin{equation}\label{eq:orden}
a\leq b\mbox{ iff } a\to b=1,
\end{equation}
\begin{equation}\label{eq:3}
a\to (b\to c)=(a\ast b)\to c=b\to (a\to c),
\end{equation}
\begin{equation}\label{eq:4}
a\to(b\wedge c)=(a\to b)\wedge (a\to c).
\end{equation}
Besides, for any BL-algebra ${\bf A}$ a unary operation of negation can be defined by the prescription $$\neg x:=x\rightarrow 0.$$
When the natural order is total ${\bf A}$ is called a {\it
	BL-chain}. As usual, a BL-algebra is called {\it complete} if for any subset $S\subseteq A$ the infimum and the supremum of $S$ exist.

BL-algebras form a variety, called $\mathbb{BL}$, which has some important subvarieties that has been studied for their own importance, since they are
the algebraic counterpart of some well known logics. We name some of them that will be used in the results of the present paper:
\begin{itemize}
	\item  \textit{MV-algebras}, the algebras of {\L}ukasiewicz infinite-valued logic, form the subvariety of $\mathbb{BL}$ characterized by the equation $\neg \neg a = a$. Along this paper, for each natural number $n$, $\mbox{\L}_n$ will denote the unique (up to isomorphism) $n$-elements MV-chain \cite{CDM}. We will see these algebras as subalgebras of the standard MV-chain $[0,1]_{\bf MV}$.
	\item  \textit{G\"odel algebras} (or prelinear Heyting algebras) are the algebraic counterpart of a superintuitionistic logic, called G\"odel Logic. The variety of G\"{o}del algebras is the subvariety $\mathbb{BL}$ that satisfy the equation $a \ast b = a \wedge b$, or equivalently $a^2:=a\ast a=a$.
	\item  \textit{Boolean algebras} are  BL-algebras that satisfy $a \wedge \neg a =0$ and $a\vee \neg a=1.$
	\item  \textit{Product algebras} form a subvariety of $\mathbb{BL}$ characterized by the following two equations: $\neg \neg c \to ((a \ast c \to b \ast c) \to (a \to b))$ and $ a \wedge \neg a = 0$. The standard product chain $[0,1]_{\bf \Pi}$ is a product algebra over the real interval $[0,1]$ interpreting $\ast$ as the standard product, $\to$ as the residuum of the product and the order is the natural order of the real unit interval. We will use the following fact \cite[Proposition 2.1]{CT}: if ${\bf A}$ is product chain, then $A\setminus \{0\}$ is a cancellative monoid, i.e., it satisfies the equation: \begin{equation}\label{eq:producto} a\to (a\ast b)=b. \end{equation}
\end{itemize}

\medskip

We recall some facts about BL-chains that will be useful to understand examples and technical lemmas along the paper.
Every BL-chain ${\bf A}$ is isomorphic to an ordinal sum indexed by a totally ordered set $I$ of algebras ${\bf A}_i$, $i \in I$  such that each ${\bf A}_i$ is either a bottom free reduct of an MV-chain or a cancellative and divisible totally ordered residuated lattice \cite[Corollary 3.2.9]{BuMon}. In symbols  $${\bf A}\cong\bigoplus_{i\in I}{\bf A}_i.$$
To define the ordinal sum we require $A_i\cap A_j=\{1\}$ for each $i\neq j\in I$ and if $\ast_i,\to_i$ are the operations in $A_i$ then the operations $\ast, \to$ in $\bigoplus_{i\in I} A_i$ satisfies \begin{align*}
		a\ast b &=\begin{cases}
				a\ast_i b & \text{if } a,b\in W_i; \\
				a          & \text{if } a\in W_i\setminus\lbrace 1\rbrace, b\in W_j \text{ and }i<j; \\
				b          & \text{if } b\in W_i\setminus\lbrace\top\rbrace, a\in W_j \text{ and }i<j.
			\end{cases}\\
		a\to b &=\begin{cases}
				1         & \text{if } a\in W_i\setminus\lbrace 1\rbrace, b\in W_j \text{ and }i<j; \\
				a\to_i b  & \text{if } a,b\in W_i; \\
				b         & \text{if } b\in W_i, a\in W_j \text{ and }i<j.
		\end{cases}
	\end{align*}
The order of $\bigoplus_{i\in I}{\bf A}_i$  can be recovered using the definition of $\to$ and equation (\ref{eq:orden}).

\section{Epistemic BL-algebras}
\label{sec:1}

To compare our algebras with the one in \cite{CCDR2017} and \cite{Bez2002}, the notation that we will adopt for the modal operators $\Box$ and $\Diamond$ is $\forall$ and $\exists$ respectively.

\begin{definition}\label{def epist alg}
	An algebra ${\bf A} = \langle A,\vee, \wedge, \ast, \to, \forall , \exists ,  0,  1 \rangle$ of type $(2, 2, 2, 2, 1, 1,  0, 0)$ is called an
	\textit{Epistemic BL-algebra} (an \textsc{EBL}-algebra for short) if $\langle A, \vee, \wedge, \ast, \to, 0, 1 \rangle$ is a \textsc{BL}-algebra that also satisfies:
		\begin{description}
			\item[(E$\forall$)] $\forall 1 =1$,
			\item[(E$\exists$)] $\exists 0 = 0$,
			\item[(E1)] $\forall a \to \exists a =1$,
			\item[(E2)] $\forall (a \to \forall b )= \exists a \to \forall b$,
			\item[(E3)] $\forall(\forall a \to b)=\forall a \to \forall b$,
			\item[(E4)] $\exists a \to \forall \exists a =1$,
			\item[(E4a)] $\forall(a \wedge b)=\forall a \wedge \forall b$,
			\item[(E4b)] $\exists (a \vee b)=\exists a \vee \exists b$,
			\item[(E5)] $\exists(a \ast \exists b)=\exists a \ast \exists b$.
		\end{description}
	
\end{definition}
Epistemic BL-algebras form a variety that we will denote by $\mathbb{EBL}$, and for simplicity, if $\mathbf{A}$ is a BL-algebra and we enrich it with an epistemic structure, we denote the resulting algebra by $\langle \mathbf{A},\forall , \exists \rangle$.
\medskip

In \cite{CCDR2017}, the authors  introduce the variety of monadic BL-algebras ($\mathbb{MBL}$) as BL-algebras endowed with two monadic operators $\forall$ and $\exists$ satisfying the equations: \begin{description}
		\item[(M1)] $\forall a \to a = 1$,
		\item[(M2)] $\forall(a \to \forall b)= \exists a \to \forall b$,
		\item[(M3)] $\forall(\forall a \to b)=\forall a \to \forall b$,
		\item[(M4)] $\forall (\exists a \vee b)=\exists a \vee \forall b$,
		\item[(M5)] $\exists( a \ast a)= \exists a \ast \exists a$.
	\end{description} They show that this class is the equivalent algebraic semantics of the monadic fragment of H\'{a}jek's basic predicate logic. The reader can corroborate that every monadic BL-algebra is an EBL-algebra, thus $\mathbb{MBL}$ is a subvariety of $\mathbb{EBL}.$ But equations M1,  M4 and M5 does not hold in any EBL-algebra. We present an example of an algebra in $\mathbb{EBL}$ which is not monadic.

\begin{example} \label{ej 1  epis no mon}
	Consider the finite MV-chain  $\mbox{\L}_4=\left \lbrace 0, \frac{1}{3}, \frac{2}{3},1 \right \rbrace$. If we define the operators $\forall$ and $\exists$ in the following way,
	
	\begin{center}	{
			\begin{tabular}{|c|c c c c|}
				\hline
				$a$        & $0$ & $\frac{1}{3}$ & $\frac{2}{3}$ & $1$ \\ \hline
				$\forall a$ & $0$ & $0$           & $1$           & $1$ \\ \hline
				$\exists a$ & $0$ & $0$           & $1$           & $1$ \\  \hline
		\end{tabular}}
	\end{center}
	then, $\langle \mbox{\L}_4, \forall, \exists \rangle$ is an epistemic BL-chain. However, this algebra does not belong to $\mathbb{MBL}$, since condition M1 is not satisfied: for $a=\frac{2}{3}$, $\forall a \not\leq a $.
\end{example}
In the next lemma we will study some properties that hold true in any EBL-algebra.

\begin{lemma}\label{propiedades}
	Let $\mathbf{A} \in \mathbb{EBL}$ and $a, b \in A$: \vspace{-.2cm}
		\begin{description}
			\item[(E6)]$\forall 0=0$,
			\item[(E7)]$\exists 1=1$,
			\item[(E8)]$\forall \forall a =\forall a$,
			\item[(E9)]$\exists \forall a=\forall a$,
			\item[(E10)]$\exists \exists a = \exists a$,
			\item[(E11)]$\forall \exists a= \exists a$,
			\item[(E12)]$\exists(\exists a \vee \exists b)=\exists a \vee \exists b$,
			\item[(E13)]$\exists(\exists a \ast \exists b)=\exists a \ast \exists b$,
			\item[(E14)]$\forall (\exists a \to b)=\exists a \to \forall b$,
			\item[(E15)]$\exists(\exists a \to \exists b)=\exists a \to \exists b$,
			\item[(E16)]$\exists(\exists a \wedge \exists b)=\exists a \wedge \exists b$,
			\item[(E17)]$\forall \neg a=\neg \exists a$,
			\item[(E18)]$\forall(\forall a \to a)=1$,
			\item[(E19)]$\forall( a \to \exists a)=1$,
			\item[(M$\forall$)] If $a \to b=1$ then $\forall a \to \forall b=1$,
			\item[(M$\exists$)] If $a \to b=1$ then $\exists a \to \exists b=1$.
		\end{description}
\end{lemma}

\begin{proof} 
	It is worth mentioning that some ideas of the following proofs are borrowed from \cite{CCDR2017}.
	\begin{description}
				\item[{(E6)}, {(E7)}] follow trivially from  {E1}, E$\forall $ and E$\exists $.
		\item [{(E8)}] From {E1} we get $\forall  \forall a\to \exists  \forall a= 1$. By {E$\forall $} and {E2}, we have $1 = \forall (\forall a\to \forall  a) = \exists  \forall a\to \forall  a$. Hence $\forall  \forall a \to \forall  a=1$. For the other direction consider {E$\forall $} and {E3}. Then  $1=\forall(\forall a\to \forall a)=\forall a \to \forall\forall a$.
		\item[{(E9)}] It is immediate from  the previous proof {E8}.
		\item[{(E10)}] Using {E5} we get  $\exists  \exists a= \exists  (1 \ast \exists  a) = \exists  1 \ast \exists  a$. By {E7} we obtain the desired result.
		\item[{(E11)}]  From {E4}, one has $\exists a\leq \forall \exists  a$. On the other hand, by {E1} and {E10}, we obtain $\forall  \exists a\leq \exists  \exists a= \exists  a$.
		\item [{(E12)}] It is an immediate consequence of {E4b}  and {E10}.
		\item[{(E13)}] It is an immediate consequence of  {E5}  and {E10}.
		\item[{(E14)}] From {E11} and {E3} we have $\forall (\exists a\to b) = \forall (\forall \exists a\to b) =   \forall \exists a\to \forall b= \exists a\to \forall  b$.
		\item[{(E15)}] Clearly, $\exists a \to \exists b= \exists a\to \forall  \exists b= \forall (\exists a\to \exists  b) \leq \exists  (\exists a\to \exists  b) $ due to {E1}, { E14} and {E1}, respectively. On the other hand, using {E11}, {E2} and {E$\forall $}, $\exists (\exists a\wedge \exists  b) \to \exists b= \exists (\exists a\wedge \exists  b) \to \forall  \exists b= \forall ((\exists a\wedge \exists  b)\to \forall  \exists  b)= \forall ((\exists a\wedge \exists  b)\to \exists  b)= \forall  1 = 1$. Thus $\exists (\exists a \wedge \exists  b)\leq \exists  b$, i.e., $\exists (\exists a\ast (\exists a\to \exists  b))\leq \exists  b$ and by {E5} $\exists (\exists a\to \exists  b) \ast \exists a\leq \exists  b$  then by residuation $\exists (\exists a\to \exists  b) \leq \exists a\to \exists  b$.
		\item[{(E16)}] By divisibility (equation (\ref{eq:divisibility})), we have $\exists  (\exists a\wedge \exists  b) = \exists (\exists a\ast (\exists a\to \exists  b))$. According to {E15}, the right side of the last identity is equivalent to $\exists (\exists a\ast \exists (\exists a\to \exists  b))$ which, due to {E13}, is equal to $\exists a\ast \exists (\exists a\to \exists  b) =\exists a\ast (\exists a\to \exists  b) = \exists a\wedge \exists  b$.
		\item[{(E17)}] Taking into account E6 and E2, $\forall \neg a=\forall (a \to 0)=\forall(a \to \forall 0)=\exists a\to \forall 0=\exists a \to 0=\neg \exists a$.
		\item[{(E18)}] Using E3, $\forall(\forall a \to a)=\forall a \to \forall a=1$.
		\item[{(E19)}] By E11, $\forall( a \to \exists a)=\forall(a \to \forall \exists a)$, then by E2, $\forall(a \to \forall \exists a)=\exists a \to \forall \exists a=\exists a \to \exists a=1$.
		\item[(M$\forall $), (M$\exists $)] They are immediate consequences of {E4a} and {E4b}, respectively.
		
	\end{description}
	
\end{proof}

\medskip
\begin{theorem} \label{subalgebra}
	If $\mathbf{A} \in \mathbb{EBL}$, then $\forall  A = \exists  A$. Moreover,  $\exists A$ is closed under the operations of ${\bf A}$, thus $\exists {\bf A}$ is a subalgebra of $\mathbf{A}$.
\end{theorem}

\begin{proof}
	Observe that from {E9} and  {E11} we have  $\forall  A =\{ \forall  a : a \in A \} = \{\exists  a : a \in A\} = \exists  A$. On the other hand, from {E$\exists$}, {E7}, {E10},{ E12}, {E13}, {E15}, and {E16}, we obtain that $\exists  \mathbf{A}$ is a subalgebra of $\mathbf{A}$.
\end{proof}

\begin{remark}
Observe that (E8) and (E10) imply that $\forall$ and $\exists$ are idempotent operators, and therefore, the identity when they are restricted to the subalgebra $\exists {\bf A}.$  (M$\forall$) and (M$\exists$) show that both operators are monotone. However, $\exists$ is not a closure operator, since  $a\le \exists a$ does not hold in every EBL-algebra, as Example \ref{ej 1  epis no mon} shows. Similarly $\forall$ is not an interior operator. \end{remark}

An \textit{implicative filter} $F$ of a BL-algebra ${\bf A}$ is a subset $F\subseteq A$ such that $1\in F$ and if $x$ and $x\rightarrow y$ are in $F$ then $y\in F.$ $F$ is also upwards closed, non-empty and closed under $\ast.$ Each filter $F$ of a BL-algebra ${\bf A}$ determines a congruence $\equiv_F$ given by $a\equiv_F b$ iff $a\rightarrow b\in F$ and $b\rightarrow a \in F.$ Moreover, the map $F\mapsto \equiv_F$ is an order isomorphism between the lattice $\mathcal{F}$ of filters of a BL-algebra ${\bf A}$ and the lattice of congruences of ${\bf A}$ \cite{HajekBook98, BuMon}. We will generalize the notion of filters for our new structures.

\begin{definition}\label{def filtros epist}
	A subset $F$ of a \textsc{EBL}-algebra $\mathbf{A}$ is an \textit{epistemic \textsc{BL}-filter} if $F$ is an implicative filter the BL-reduct of ${\bf A}$ and  if $a \to b \in F$ then $\forall  a \to \forall  b \in F$ and $\exists  a \to \exists  b \in F$.
\end{definition}

\begin{theorem} \label{EBL filtros congruencias}
	Let $F$ be an epistemic \textsc{BL}-filter of an \textsc{EBL}-algebra $\mathbf{A}$. Then the binary relation $\equiv_F$ on $A$ defined by $a \equiv_F b$ if and only if $a \to b \in F$ and $b \to a \in F$ is a congruence relation. Moreover, $F = \{ a \in A :  a \equiv_F 1 \}$. Conversely, if $\equiv$ is a congruence on  $A$, then $F_\equiv =\{ a \in A :  a \equiv 1 \}$ is an epistemic  \textsc{BL}-filter, and $a \equiv b$ if and only if $a \to b \equiv 1$ and $b \to a \equiv 1$. Therefore, the correspondence $F \mapsto \equiv_F$ is a bijection from the set of epistemic  \textsc{BL}-filters of $\mathbf{A}$ onto the set of congruences on $\mathbf{A}$.
\end{theorem}

\begin{proof}
The fact that the congruence $\equiv_F$ of a BL-algebra ${\bf A}$ is also a congruence of the EBL-algebra ${\bf A}$ follows immediately from the definition of epistemic \textsc{BL}-filter.
 We will check that  $F = \{ a \in A :  a \equiv_F 1 \}$. In details,  $a \to 1=1 \in F$ and if $a \in F$, since $a = 1 \to a$, we have $1 \to a \in F$. Hence $a \equiv_F 1$. On the other hand, if we consider $a \in \{ a \in A :  a \equiv_F 1 \}$, is immediate that $a = 1 \to a \in F$.

	To complete the proof, assume now that $\equiv$ is a congruence on $\mathbf{A}$, in which case the quotient algebra ${\mathbf{A}}/{\equiv}$ is also a \textsc{EBL}-algebra. We shall denote by $\{[a]:a\in A\}$ the set of equivalence classes of ${\mathbf{A}}/{\equiv}.$ It is already known that $F_\equiv$ is a $BL$-filter. Let us see that $F_{\equiv}=\{ a \in A :  a \equiv 1\}$ is an epistemic $BL$-filter. Consider $a,b\in A $ such that $a \to b \in F_{\equiv}$. Then by $M\forall $ we have:
	\begin{center}
		\begin{tabular}{c c c }
			$ a \to b \equiv 1$ & $\Leftrightarrow$ & $[a \to b] = [1]$\\
			& $\Leftrightarrow$ & $[a] \to [b] = [1]$\\
			& $\Rightarrow$ & $\forall  [a] \to \forall  [b] =[1]$\\
			& $\Leftrightarrow$ & $ [\forall  a] \to  [\forall  b] = [1]$\\
			& $\Leftrightarrow$ & $[\forall  a \to \forall  b] = [1]$\\
			& $\Leftrightarrow$ & $\forall  a \to \forall  b \equiv 1$\\
			& $\Leftrightarrow$ & $\forall  a \to \forall  b \in F$\\
		\end{tabular}
	\end{center}
	Analogously for $\exists  a \to \exists b\in F$, using property $M\exists $. Hence, $F_\equiv$ is \textsc{EBL}-filter and is also true that $a \equiv b$ if and only if $a \to b \equiv 1$ and $b \to a \equiv 1$.
\end{proof}

\section{Subvarieties of EBL-algebras}
\label{sec:2}

In this section we will see that EBL-algebras generalize two well-studied classes of algebras. In particular, we are interested in subclasses that are algebraic counterparts  of modal fuzzy logic KD45 systems. First, given an epistemic BL-algebra $\left\langle {\bf A}, \forall, \exists \right\rangle$, we will show that when the reduct $A$  is a Boolean algebra, the system  $\langle {\bf A}, \exists \rangle$ is a Pseudomonadic algebra in the sense of \cite{Bez2002}. On the other hand, if the BL-reduct of {\bf A} is a G\"odel structure, we proof that the resultant algebra is a Bi-modal G\"odel algebra defined by \cite{CaiRod2015}.
\subsection{Pseudomonadic Algebras}
\label{sec:3}
In \cite{Bez2002} the author introduce the class of Pseudomonadic algebras as natural generalization of Halmos' monadic algebras and shows that they serve as algebraic models of KD45 over classical logic.
\begin{definition}\label{def pseudomonadicas}\cite[Definition 2.1]{Bez2002} An algebra $\langle {\bf B}, \exists \rangle$ is said to be a \textit{Pseudomonadic algebra} if ${\bf B}$ is a boolean algebra and $\exists$ is a unary operator on $B$ satisfying the following identities for every $a, b \in B$:
	\begin{description}
		\item[(P1)]$\exists 0 =0$,
		\item[(P2)]$\exists(a \vee b)=\exists a \vee \exists b$,
		\item[(P3)]$\exists(\exists a \wedge b)=\exists a \wedge \exists b$,
		\item[(P4)]$\neg \exists a \leq \exists \neg a$.
	\end{description}
\end{definition}
The class of Pseudomonadic algebras form a variety which is denoted by $\mathbb{PMA}$ and it is a proper extension of Halmos' variety of monadic algebras. In this case we use $\forall$ as abbreviation of the operator $\neg \exists \neg$.\\
The next lemma lists some properties that hold true in any Pseudomonadic algebra that we will use in the main theorem.

\begin{lemma}\cite[Lemma 2.1]{Bez2002} \label{Prop pseudomonadicas} The following identities hold in every Pseudomonadic algebra:
		\begin{description}
			\item[(P5)] $\forall a \leq \exists a$,
			\item[(P6)] $\forall 1=1$,
			\item[(P7)] $\exists 1=1$,
			\item[(P8)] $\forall 0=0$,
			\item[(P9)] $\exists \exists a= \exists a$,
			\item[(P10)] $\forall \forall a= \forall a$,
			\item[(P11)] $\forall \exists a =\exists a$,
			\item[(P12)] $\exists \forall a = \forall a$,
			\item[(P13)] $\exists \neg \exists a = \neg \exists a$,
			\item[(P14)] $\forall \neg \forall a =\neg \forall a$,
			\item[(P15)] $\forall (\forall a \vee b)=\forall a \vee\forall b$,
			\item[(P16)] $\forall(a \wedge b)=\forall a \wedge \forall b$,
			\item[(P17)] $\forall(a \to b) \leq \forall a \to \forall b$,
			\item[(P18)] $\exists( \exists a \to a)=1$,
			\item[(P19)] $\forall(\forall a \to a)=1$.
		\end{description}
\end{lemma}

\begin{theorem} \label{epis=pseud}
	Let ${\bf A}$ be a Boolean algebra. Then $\langle {\bf A}, \forall, \exists \rangle$ is an EBL-algebra iff $\langle {\bf A},\exists \rangle$ is a Pseudomonadic algebra.
\end{theorem}
\begin{proof}
	Assume first that $\langle \bf{A}, \exists, \forall \rangle$  is an  EBL-algebra and recall that since ${\bf A}$ is Boolean the equation $\neg\neg a=a$ holds in A. Then:
	\begin{description}
		\item[(P1)]$\exists 0= 0$ is satisfied by E$\exists$.
		\item[(P2)]$\exists (a \vee b)= \exists a \vee \exists b$ it is immediate from E4b.
		\item[(P3)] $\exists(\exists a \wedge b )= \exists a \wedge \exists b$, it is immediate from  E5.
		\item[(P4)] $\neg \exists a \leq \exists \neg a$. Indeed, applying E17 and E1, we have $\neg \exists a = \forall \neg a \leq \exists \neg a$.
	\end{description}
	Conversely, suppose that $\langle {\bf A}, \exists, \forall \rangle$ is a Pseudomonadic algebra. Properties (E$\forall$), (E$\exists$), (E1), (E4), (E4a), (E4b) and (E5) are immediate from P6, P1, P5, P13, P16, P2 and P3, respectively. Let's see that (E2) and (E3) are also satisfied:
	\begin{description}
		\item[(E2)] $\forall (a \to \forall b)=\exists a \to \forall b$. Considering P3 and the equivalence $\forall a = \neg \exists \neg a$, we have
		\begin{align*}
		\forall (a \to \forall b)&=\neg \exists \neg (\neg a \vee \forall b)\\
		&=\neg \exists (a \wedge \neg \forall b)\\
		&=\neg \exists (a \wedge \exists \neg b)\\
		&=\neg (\exists a \wedge \exists \neg b)\\
		&=\neg \exists a \vee \neg \exists \neg b \\
		&=\neg \exists a \vee \forall b \\
		&= \exists a \to \forall b
		\end{align*}
		\item[(E3)]$\forall(\forall a \to b)=\forall a \to \forall b$. Applying P12 of Lemma \ref{Prop pseudomonadicas} and P3,
		\begin{align*}
		\forall (\forall a \to b)&=\forall (\neg \forall a \vee b)\\
		&=\forall \neg(\forall a \wedge \neg b)\\
		&=\neg \exists (\forall a \wedge \neg b)\\
		&=\neg\exists (\exists \forall a \wedge \neg b)\\
		&= \neg(\exists \forall a \wedge \exists \neg b)\\
		&=\neg (\forall a \wedge \neg \forall b)\\
		&=\neg \forall a \vee \forall b \\
		&=\forall a \to \forall b
		\end{align*}
	\end{description}
\end{proof}

To characterize simple, subdirectly irreducible and well-connected Pseudomonadic algebras, Bezhanishvili defines a special type of filters that are in correspondence with congruences of  Pseudomonadic algebras.

\begin{definition}\label{def forall filtros}  Given a Pseudomonadic algebra ${\bf B}$, a subset $F \subseteq B$ is a \textit{$\forall$-filter} of ${\bf B}$ if $F$ is a filter and  if $a \in F$ implies $\forall a \in F$, for each $a \in B$.
\end{definition}

\begin{theorem} Let $\langle \bf{A}, \exists, \forall \rangle$ be an EBL-algebra such that $\bf{A}$ is Boolean and let $F$ be a filter of $A$.  Then $F$ is an epistemic filter iff $F$ is a $\forall$-filter.
\end{theorem}
\begin{proof}
	Suppose that $F$ is an epistemic filter and let $a \in A$ be such that $a \in F$. Since $a = 1 \to a$, we have $1 \to a \in F$, with which  $\forall 1 \to \forall a \in F$, i.e., $\forall a \in F$.\\
	Conversely, suppose that $F$ is a $\forall$-filter and let $a, b \in a$ such that $a \to b \in F$. By definition $\forall(a \to b) \in F.$
	Moreover, by Theorem \ref{epis=pseud} and property P17, we have $$\forall(a \to b) \leq \forall a \to \forall b$$
	with which, $\forall a \to \forall b \in F$.\\
	On the other hand, since in every Boolean algebra the inequation $a \to b \leq \neg b \to \neg a$ holds, then $\neg b \to \neg a \in F$. Being $F$ a $\forall$-filter, we get  $\forall (\neg b \to \neg a) \in F$.\\
Taking into account property P17 from Lemma \ref{Prop pseudomonadicas}, we have $\forall \neg b \to \forall \neg a \in F$, or equivalently, $\neg \exists b \to \neg \exists a \in F.$
	Therefore $\exists a \to \exists b \in F$.
\end{proof}

\subsection{Bi-modal G\"odel Algebras}
\label{sec:4}
 G\"odel algebras can be characterized as the subvariety of Heyting algebras determined by the prelinearity equation (\ref{eq:prelinearity}). They also can be thought of as the  subvariety of BL-algebras that satisfy the equation $a^2=a$, or equivalently  $a \ast b = a \wedge b$.\\

\medskip

\cite{CaiRod2015} show that the set  of valid formulas in the subclass of serial, transitive and euclidean GK-frames (G\"odel-Kripke frames) is axiomatized by adding some additional axioms and a rule to those of G\"odel fuzzy logic $G.$
The logic obtained is denoted $KD45(G)$ and has as algebraic semantics the variety of \textit{Bi-modal G\"odel algebras}. These algebras, of the form $\langle {\bf G}, \forall, \exists \rangle$ \footnote{ \cite{CaiRod2015} use the notations $I$ and $K$ for $\forall$ and $\exists$, respectively.}, are such that ${\bf G}$ is a G\"odel algebra and $\forall$, $\exists$ are unary operators over $G$  that satisfy the following identities:
	\begin{description}
		\item[(G1)] $\forall (a \ast b) = \forall a \ast \forall b$,
		\item[(G2)] $\forall 1=1$,
		\item[(G3)] $ \exists a \to \forall b \leq \forall (a \to b)$,
		\item[(G4)] $\exists( a \vee b) = \exists a \vee \exists b$,
		\item[(G5)] $\exists 0=0$,
		\item[(G6)] $\exists(a \to b) \leq \forall a \to \exists b$,
		\item[(G7)] $\forall a \leq \exists a$,
		\item[(G8)] $\forall a \leq \forall \forall a$ \hspace{1cm} $\exists a \leq \exists \exists a$,
		\item[(G9)] $\exists a \leq \forall \exists a$ \hspace{1cm} $\exists \forall a \leq \forall a$.
	\end{description}
$KD45(G)$ is complete with respect to valuations in these algebras. We shall see that the class of epistemic G\"odel algebras  and the one of Bi-modal G\"odel algebras coincide.

\begin{theorem}
	If $\langle {\bf A}, \forall, \exists \rangle$ is an epistemic algebra such that ${\bf A}$  is a G\"odel algebra, then $\langle {\bf A}, \forall, \exists \rangle$ is a bi-modal algebra.
\end{theorem}

\begin{proof} Let ${\bf A}$ be a G\"odel algebra such  that $\langle {\bf A}, \forall, \exists \rangle$ satisfy axioms of Definition \ref{def epist alg}. Let's check that G1-G9 are satisfied.
\smallskip
(G1), (G2), (G4), (G5) and (G7) are immediate from E4a,  E$\forall$, E4b, E$\exists$ and E1 respectively. To check G3 and G6 we will use the auxiliary result: \begin{equation}\label{eq:aux} \forall(a \to b)\leq \forall a \to \forall b\end{equation} that holds in ${\bf A}.$ Indeed, $a \ast b =a \wedge b = a \ast (a \to b) \leq b$, then by monotonicity of $\forall$, we have $\forall (a \ast (a \to b)) \leq \forall b$. Applying G1 and residuation we obtain $\forall (a \to b) \leq \forall a \to \forall b$.

\smallskip

We now check (G3), that is, $ \exists a \to \forall b \leq \forall (a \to b)$.  Since $b \leq a \to b$, by M$\forall$, we have $\forall b \leq \forall(a \to b)$. Then
		\begin{align*}
		1 &= \forall b \to \forall(a \to b)\\
		&=\forall(\forall b \to (a \to b))  & E3\\
		&= \forall(a \to (\forall b \to b))& (\ref{eq:3})\\
		&\leq\forall((a \to \forall b) \to (a \to b)) & \mbox{(\ref{eq:residuation}) and M$\forall$}\\
		&\leq \forall(a \to \forall b) \to \forall(a \to b) & (\ref{eq:aux})\\
		&=(\exists a \to \forall b) \to \forall(a \to b) &E2\\
		\end{align*}
Hence $\exists a \to \forall b \leq \forall(a \to b).$ To see that (G6) holds, that is, $\exists(a \to b) \leq \forall  a \to \exists b$, we first see that E11 and E2 imply
$$1 = \exists b \to \exists b= \exists b \to \forall \exists b = \forall (b \to \forall \exists b) = \forall (b \to \exists b).$$ The residuation condition yield $$b\to \exists b\le a\to (b\to \exists b)\le (a\to b)\to (a\to \exists b),$$ thus
		\begin{align*}
		1 &= \forall (b \to \exists b) \\
		& \leq\forall((a \to b)\to (a \to \exists b)) & \mbox{M$\forall$}  \\
		&=\forall( a \to ((a\to b) \to \exists b)) & (\ref{eq:3})\\
		&\leq \forall a \to \forall ((a\to b) \to \exists b) & (\ref{eq:aux})\\
		&= \forall a \to \forall ((a\to b) \to \forall \exists b) & E1\\
		&=\forall a \to (\exists(a \to b) \to \forall \exists b) & E2\\
		&=\forall a \to (\exists(a \to b) \to  \exists b) & E11\\
		&=\exists (a \to b) \to (\forall a \to \exists b).
		\end{align*}
Then $\exists(a \to b) \leq \forall a \to \exists b$. Lastly (G8) is immediate from E8 and E10 while (G9) follows from E4 and E9.
		\end{proof}

\begin{theorem} If $\langle {\bf A}, \forall, \exists \rangle$ is a bi-modal G\"odel algebra for KD45, then $\langle {\bf A}, \forall, \exists \rangle$ is an epistemic G\"odel algebra.
\end{theorem}
\begin{proof}
	Let $\langle {\bf A}, \forall, \exists \rangle$ be a bi-modal G\"odel algebra, i.e., it satisfies G1-G9. Properties (E$\forall$), (E$\exists$), (E1), (E4), (E4a) and (E4b) are immediate from G2, G5, G7, G9, G1 and G4, respectively. The remaining properties require some calculations. To proof E2, let's see that
	\begin{equation}\label{eq:aux2}
	\forall(a\to b)\leq \exists a \to \exists b
	\end{equation}
	is satisfied. Indeed, since $a \ast (a \to b) \leq b$, by (\ref{eq:residuation}), we have $a  \leq (a \to b) \to b$. Then, applying G4 and G6, we obtain $\exists a \leq \exists((a \to b) \to b) \leq \forall (a \to b) \to \exists b$. Consequently, $\forall(a \to b)\leq \exists a \to \exists b$.\\
	Now we are able to check the properties E2, E3 and E5:
	\begin{description}
	    \item[(E2)] $\forall(a \to \forall b)=\exists a \to \forall b$.
	    By G8  and G3 we get  $\exists a \to \forall b \leq \exists a \to \forall \forall b \leq \forall(a \to \forall b)$.
	    For the other inequality,
	    \begin{align*}
	    \forall( a \to \forall b) &\leq \exists a \to \exists \forall b &(\mbox{\ref{eq:aux2}})\\
	    &\leq \exists a \to \forall b. &G9
	    \end{align*}
	     Therefore  $\forall(a \to \forall b)=\exists a \to \forall b$.
	
	    \item[(E3)]  $\forall(\forall a \to b) =\forall a \to \forall b$. Since ${\bf A}$ satisfies E4a, then M$\forall$ also holds in ${\bf A}$. This monotonicity, together with G1 and residuation imply that equation (\ref{eq:aux}) holds in ${\bf A}$. Then \begin{align*}
	    \forall(\forall a \to b) &\leq \forall \forall a \to \forall b & (\ref{eq:aux}) \\
	    &\leq \forall a \to \forall b &G8
	    \end{align*}
	
	    Hence $\forall(\forall a \to b) \leq \forall a \to \forall b.$ For the other inequality G9 and an application of G3 yield, $$\forall a \to \forall b \leq \exists \forall a \to \forall b \leq \forall(\forall a \to b).$$
		\item[(E5)]$\exists(a \ast \exists b)=\exists a \ast \exists b$.
		\begin{align*}
		a \ast \exists b &= a \wedge \exists b \\
		a  & \leq \exists b \to (a \wedge \exists b)  &(\mbox{\ref{eq:residuation}})\\
		\exists a & \leq \exists (\exists b \to (a \wedge \exists b)) & G4\\
		\exists a & \leq \forall \exists b \to \exists (a \wedge \exists b) & G6\\
		\exists a & \leq  \exists b \to \exists (a \wedge \exists b) & G9\\
		\exists a \ast \exists b & \leq  \exists (a \wedge \exists b) & (\mbox{\ref{eq:residuation}})\\
		\end{align*}
		With which $\exists a \wedge \exists b \leq \exists(a \wedge \exists b)$.\\
		On the other hand, since $a \wedge \exists b \leq a$, applying G4, $\exists(a \wedge \exists b)\leq \exists a$. Similarly, $a\wedge \exists b \leq \exists b$, then, by monotony and G8, $\exists (a \wedge \exists b)\leq \exists \exists b \leq \exists b$.\\
		As consequence, $\exists(a \wedge \exists b) \leq \exists a \wedge \exists b$.
	
	\end{description}

	\end{proof}
\section{c-EBL-algebras}
\label{sec:5}
We will now introduce a special class of epistemic algebras that will become important to establish a connection with possibilistic frames in the next section. Given an epistemic BL-algebra $\langle{\bf A}, \forall, \exists\rangle$, if the set $$\{a \in A: \forall a=1\}\subsetneq A$$ has a least element $c$, then $c$  will be called \emph{focal element} of {\bf A}. For example, if we consider the epistemic BL-algebra given in Example \ref{ej 1  epis no mon} we have that $\min\{a \in A: \forall a=1\}=\frac{2}{3}$, i.e., the focal element exists and it is $\frac{2}{3}$. In fact, it is immediate that if ${\bf A}$ is an EBL-algebra such that $A$ is finite, the focal element exists. But this is not the case of every EBL-algebra, as we can see in the following example: consider the BL-algebra $\bf{A}=[0,1]_{MV}\oplus[0,1]_{\Pi}$, where $\oplus$ is the ordinal sum operation. Let's denote $0_{\Pi}$ the zero element in the second summand. Defining the operators $\forall$ and $\exists$ by:
	    \begin{center}
	    $\forall x = \left\{
		\begin{array}{l l}
		1 & x \in (0,1]_{\Pi}\\
		0_{\Pi}  & x= 0_{\Pi}\\
		0 & x \in [0,1)_{MV}
		\end{array}
		\right.$
		$ \exists x = \left\{
		\begin{array}{l l}
		1 & x \in (0,1]_{\Pi}\\
    	0_{\Pi} &  x \in (0,1)_{MV} \cup \{0_{\Pi}\}\\
		0 & x=0
		\end{array}
		\right.$
		\end{center}
we obtain a structure of EBL-algebra such that the set $\{a \in A: \forall a=1\}=(0,1]_{\Pi}$ has no a least element.

\begin{remark} The focal element of an EBL-algebra satisfies
	\begin{equation} \label{rem c=min}
	c = \min_{ a \in A}\{(\forall  a \to a) \wedge (a \to \exists  a)\}.
	\end{equation}
\end{remark}
\begin{proof} Let $x$ be an element of the form $x =(\forall  a \to a) \wedge (a \to \exists  a)$ for some $a \in A$. Then, taking into account E4a, E3, E11 and E2, we have
	\begin{align*}
	\forall x &=\forall((\forall  a \to a) \wedge (a \to \exists  a))\\
	&=\forall(\forall  a \to a) \wedge \forall(a \to \exists  a)\\
	&=(\forall a \to \forall a ) \wedge (\exists a \to \exists a)\\
	&=1.
	\end{align*}
	So, $x \in \{x \in A: \forall x=1\}$. Hence, by definition of focal element, $c \leq x$.
	On the other hand, as $\forall c=\exists c=1$, we can write $c= (\forall c \to c) \wedge (c \to \exists c)$. Therefore, $c$ is a least element of $\{ (\forall  a \to a) \wedge (a \to \exists  a), a \in A\}$.
\end{proof}
\begin{definition}\label{c-EBL}
	An EBL-algebra ${\bf A}$, we will be called a \emph{c-EBL-algebra}, if the focal element $c$ exists in $A$.
\end{definition}


On the class of c-EBL-algebras, the focal element plays a fundamental role, since it allows us to recover the unary operators $\forall$ and $\exists$, as the following theorem shows.

\begin{theorem}\label{c-teorema}
	Let {\bf A} be a c-EBL-algebra and let $B$ be the subalgebra given by Theorem \ref{subalgebra}, then
	\begin{description}\item $\forall a = \max\{b \in B: b \leq c \to a\}$  and  \item   $\exists  a = \min\{b \in B: c\ast a \leq b\}.$ \end{description}
\end{theorem}
\begin{proof}
	Since $B=\forall  a=\exists  a$, then $\forall  a \in B$ and $\exists  a \in B$ for every $a \in A$. On the other hand, since $c$ satisfies (\ref{rem c=min}), we have
	$$c \leq (\forall  a \to a) \wedge (a \to \exists  a)\leq \forall  a \to a$$
	for all $a \in A$. By residuation $\forall  a \leq c \to a.$
	Now assume there is $b \in B$ such that $b \leq c \to a$. Since $B=\forall  A$, there exists $x \in A$ such that $b = \forall  x$, and
	$\forall  x \leq c \to a$, or equivalently, $$c \leq \forall  x \to a.$$
	Taking into account the properties M$\forall$ and E3, we obtain
	$\forall  c \leq \forall  x \to \forall  a.$
	Since $\forall  c = 1$, it follows that $b=\forall  x \leq \forall  a$ and therefore $\forall  a = \max\{b \in B: b \leq c \to a\}$.\\
	Arguing as above, $$c \leq (\forall  a \to a) \wedge (a \to \exists  a)\leq a \to \exists  a$$
	for every $a \in A$. Thus $c \ast a \leq \exists  a.$
	Suppose there exists $b \in B$ such that $c \ast a \leq b$ with $b =\exists  y$ for some $y \in A$. By residuation $c \leq a \to \exists  y.$
	Applying properties M$\forall$, E11 and E2, we obtain $\forall  c \leq \exists  a \to \exists  y$, and then
	$$\exists  a \leq \forall  c \to \exists  y =1 \to \exists  y = \exists  y = b.$$
	We can conclude $\exists  a = \min\{b \in B: c\ast a \leq b\}.$
	
\end{proof}

According to Theorem \ref{subalgebra}, if $\langle \textbf{A}, \forall , \exists \rangle$ is an EBL-algebra, then $\exists A= \forall  A$ is a BL-subalgebra of $ \textbf{A}$.
We are going to show under which conditions an epistemic BL-algebra can be defined from a BL-algebra $ \textbf{A}$  and one of its subalgebras  $\textbf{B}.$

\begin{definition}\label{def c-rel-comp subal}
	Let {\bf A} be a BL-algebra, {\bf B} a subalgebra of \textbf{A} and $c\in A$. We say that the pair $(B, c)$ is a  \emph{$c$-relatively complete subalgebra}, if the following conditions hold:
	\begin{description}
		\item[(e1)] For every $a \in A$, the subset $\{b \in B : b \leq c \to a\}$ has a greatest element and the subset $\{b \in B : c \ast a \leq b\}$ has a least element.
		\item[(e2)]  $\{a \in A: c^2 \leq a \} \cap B = \{1\}$.
	\end{description}
\end{definition}

\begin{theorem} \label{teo c-rel-completa}
	Given a \textsc{BL}-algebra \textbf{A} and a $c$-relatively complete subalgebra $(\textbf{B} ,c)$, if we define on \textbf{A} the operations:
	\begin{equation}\label{def forall sobre c}
	\forall  a := \max\{b \in B : b \leq c \to a\},
	\end{equation}
	\begin{equation}\label{def exists sobre c}
	\exists  a := \min\{b \in B : c \ast a \leq b\},
	\end{equation}
	then $\langle \textbf{A}, \forall , \exists \rangle$ is a c-EBL-algebra such that $\forall  A = \exists A = B$. Conversely, if \textbf{A} is a c-EBL-algebra, then $(\forall  A, c)$ is a $c$-relatively complete subalgebra of \textbf{A}.
\end{theorem}
\begin{proof}
	Clearly condition (e1)  guarantees the existence of $\forall  a$ and $\exists  a$ for every $a \in A$.
	It remains to show that $\langle A, \forall , \exists  \rangle$ satisfies Definition \ref{def epist alg}. Let $a,b \in A$:
	\begin{description}
		\item[(E$\forall $)] It is clear that $\forall  1 = \max\{b \in B : b \leq c \to 1\} = 1$ because $B$ is subalgebra.
		\item[(E$\exists $)] In the same sense, clearly, $\exists  0 = \min\{b \in B : c \ast 0 \leq b\} = 0$.
		\item[(E1)] First note that, by (\ref{def exists sobre c}) and (e2), we have $\exists  c = 1$. Let us check that $\forall  a \leq \exists  a$.
		By definition, we have $\forall  a \leq c \to a$. Then $c \ast \forall  a  \leq a$ and $c^2 \ast \forall  a \leq c \ast a$. As $c \ast a \leq \exists  a$, we get $c^2 \ast \forall  a \leq \exists  a.$ By residuation, $$c^2 \leq \forall  a \to \exists  a.$$
		Besides, as $\exists  c = \min \{b \in B : c^2 \leq b\}= 1$ and  $B$ is subalgebra ($\forall  a \to \exists  a \in B$), we obtain $$1=\exists  c \leq \forall  a \to \exists  a,$$
		and $\forall  a \to \exists  a = 1$ as desired.
		\item[(E2)] From $a \ast c \leq \exists  a$, we get
		$$\exists  a \to \forall  b \leq (a \ast c) \to \forall  b = c \to (a \to \forall  b).$$
		Since $\exists  a \to \forall  b \in B$, we have $$\exists  a \to \forall  b \leq \forall  (a \to \forall  b). $$
		On the other hand, we know that $\forall (a \to \forall  b)\leq c \to (a \to \forall  b)= (a \ast c) \to \forall  b$, thus, by residuation
		$$a \ast c \leq \forall (a \to \forall  b) \to \forall  b.$$
		Then, taking into account (\ref{def exists sobre c}), we obtain $\exists  a \leq \forall (a \to \forall  b) \to \forall  b$, so
		$$ \forall (a \to \forall  b) \leq \exists  a \to \forall  b.$$
		
		\item[(E3)]Again by definition, $\forall  (\forall  a \to b) \leq c \to (\forall  a \to b)= \forall  a \to (c \to b)$, or equivalently,
		$$\forall  (\forall  a \to b) \ast \forall  a \leq c \to b$$
		By (\ref{def forall sobre c}), we obtain $$\forall  (\forall  a \to b) \leq \forall  a \to \forall  b.$$
		Moreover, $\forall  b \leq c \to b$ implies $$\forall  a \to \forall  b \leq \forall  a \to (c \to b)=c \to (\forall  a \to b),$$
		from where, by definition of $\forall (\forall  a \to b)$, we get
		$$\forall  a \to \forall  b \leq \forall (\forall  a \to b).$$
		
		\item[(E4)] We know that $\exists  a \leq c \to \exists  a$ and $\exists  a \in B$, therefore $\exists  a \leq \forall  \exists  a$, i.e., $\exists  a \to \forall  \exists  a =1.$
		\item[(E4a)] Note that $\forall (a \wedge b)\leq c \to (a \wedge b)= (c \to a) \wedge (c \to b)$, then
		$$\forall (a \wedge b) \leq c \to a \hspace{1.5cm}\forall (a \wedge b) \leq c \to b $$
		As consequence,
		$$\forall (a \wedge b) \leq \forall  a \hspace{1.5cm}\forall (a \wedge b) \leq \forall  b, $$
		and hence $$\forall (a \wedge b)\leq \forall  a \wedge \forall  b.$$
		On the other hand, $\forall  a \wedge \forall  b \leq \forall  a \leq c \to  a$ and $\forall  a \wedge \forall  b \leq \forall  b \leq  c \to  b$.
		Thus, $\forall  a \wedge \forall  b  \leq (c \to  a) \wedge (c \to b)= c \to (a \wedge b)$.\\
		Since $\forall  (a \wedge b) = \max\{a \in B: a \leq c \to (a \wedge b)\}$, we obtain $$\forall  a \wedge \forall  b \leq \forall  (a \wedge b).$$
		
		\item[(E4b)]The proof of this case is analogous to the previous one.
		\item[(E5)] By definition $c \ast (a \ast \exists  b)\leq \exists (a \ast \exists  b)$, then by residuation $c \ast a \leq \exists  b \to \exists (a \ast \exists  b)$. So $$\exists  a \leq \exists  b \to \exists (a \ast \exists  b),$$ or equivalently $\exists  a \ast \exists  b \leq \exists (a \ast \exists  b).$
		Since $a \ast c \leq \exists  a$, we have $(a \ast c)\ast \exists  b \leq \exists  a \ast \exists  b$, therefore
		$$\exists (a \ast \exists  b) \leq \exists  a \ast \exists  b.$$
	\end{description}
	Thus $\langle \textbf{A}, \forall , \exists \rangle$ is an epistemic BL-algebra.
\smallskip

	We now verify that $c$ is the focal element of  $\langle \textbf{A}, \forall , \exists \rangle.$ To that aim, let $a$ be an element of $\{a \in A: \forall a=1\}$. Then from  (\ref{def forall sobre c}), we get $1=\forall a \leq c \to a$ and $c \leq a$. Besides, $\forall c=\max\{b \in B: b \leq c \to c\}=1$. Therefore, $c=\min\{a \in A: \forall a=1\}$.\\
	Let us see now that $\forall  A = \exists  A = B$. By the previous and Theorem \ref{subalgebra} the first equality is satisfied. On the other hand, it's clear that $\exists  A \subseteq B$. Furthermore, for all $b \in B$, $c \ast b \leq b $, whereby $\exists  b \leq b$. Besides $c^2 \ast b \leq c \ast b \leq \exists  b$, then, by residuation, $c^2 \leq b \to \exists  b$ but as $b , \exists  b \in B$ and $B$ is subalgebra it follows that $b \to \exists  b \in B$ and is greater than $c^2$, thus $1 = \exists  c \leq b \to \exists  b$. Consequently $b \leq \exists  b$ and hence $B \subseteq \exists  A$.\\
	
	Conversely, let $\langle \textbf{A}, \forall , \exists \rangle$ be a c-EBL-algebra. From Theorem \ref{subalgebra}, we know that $\exists  \mathbf{A}$ is a BL-subalgebra of $ \textbf{A}$. Let us now show that conditions (e1) and (e2) hold.
	\begin{description}
		\item[e1] By Theorem \ref{c-teorema} $\forall  a = \max\{b \in \forall  A: b \leq c \to a\}$ and $\exists  a = \min \{b \in \forall  A: c \ast a \leq b\}$.
		\item[e2]  $\{a \in A: c^2\leq a\}\cap \forall A =\{a \in \forall A: c^2\leq a\}$. By Theorem \ref{c-teorema}, the set $\{a \in \forall A: c^2\leq a\}$ has a least element and it is  $\exists  c$. Therefore, $\{a \in \forall A: c^2\leq a\}=\{1\}$.
	\end{description}
\end{proof}

\section{H\'{a}jek's fuzzy modal logic $KD45(\mathcal{C})$}
\label{sec:6}

In \cite{Hajeketal1995}, the authors define a modal logic to reason about possibility and necessity degrees of many-valued propositions. The fuzzy modal belief logic, that they call  $KD45(\mathcal{C})$, is defined as the set of valid formulas in the class of $\mathcal{C}$-possibilistic frames. This logic is a generalization of the so-called {\it Possibilistic Logic} \cite{Duboisetal1,Duboisetal2} a well-known uncertainty logic to reasoning with graded belief on classical propositions by means of necessity and possibility  measures. In his book, H\'{a}jek leaves as an open problem to find an elegant axiomatization for $KD45(\mathcal{C})$ with models over BL-algebras  $\mathcal{C}$ \cite[Chapter 8, pages 227-228]{HajekBook98}. Some attempts to solve the open problem have been done, considering that the most natural semantics of fuzzy $KD45$ is a fuzzy version  of possibilistic frames (closely related to the Kripke-style semantics),  see for example \cite{BoEsGoRo11} and the reference therein.
Our aim is to give a more general characterization of the logical system by our algebraic approach. Nevertheless, we establish a connection between our Epistemic BL-algebras and BL-possibilistic frames, trying to be closer to an answer to H\'{a}jek's open question.


\subsection{Complex algebras}
\label{sec:7}

Given a complete \textsc{BL}-algebra ${\bf A}$, a \emph{possibilistic} ${\bf A}$\emph{frame} ($\Pi {\bf A}$-frame) is a
structure $\langle W,\pi \rangle $ where $W$ is a non-empty set of worlds, and $\pi:W \to {\bf A}$ (i.e. $\pi \in {\bf A}^{W}$) is a normalized possibility distribution over $W$, i.e., $\sup_{w \in W} \pi(w) = 1$.\\
For a fix $\Pi {\bf A}$-frame $\mathcal{P}=\langle W, \pi \rangle$, considering the BL-algebra ${\bf A}^{W}$, of functions from $W$ to ${\bf A}$ and operations $(\ast, \to, \wedge, \vee, 0,1)$ defined pointwise. For each element $f \in A^{W}$ let:
\begin{equation}\label{poss-epis1}
\forall ^{\mathcal{P}}(f) =  \inf_{w \in W}\{\pi(w) \to f(w) \},
\end{equation}
\begin{equation}\label{poss-epis2}
\exists ^{\mathcal{P}}(f)  = \sup_{w \in W}\{ \pi(w) \ast f(w)\}.
\end{equation}
Observe that these operators are well-defined, since the algebra ${\bf A}$ is complete. The system $\langle {\bf A}^{W}, \forall ^{\mathcal{P}}, \exists ^{\mathcal{P}}\rangle$ will be called the \textit{complex ${\bf A}$-algebra} associated with the $\Pi {\bf A}$-frame $\mathcal{P}=\langle W, \pi \rangle$.

We shall prove that $\langle {\bf A}^{W}, \forall ^{\mathcal{P}}, \exists ^{\mathcal{P}} \rangle$ is an EBL-algebra. To achieve so, we need the next auxiliary result.

\begin{lemma} \label{existsc=1}
	$\exists ^{\mathcal{P}}(\pi)=1$, i.e., $\sup_{w \in W}\{\pi(w)^2\}=1$. 
\end{lemma}

\begin{proof}
	The result of the lemma will follow from the fact that $\sup_{w \in W}\{\pi(w)\}=1$ and the next equality:  $$\sup_{w \in W}\{\pi(w)^2\} =  \sup_{w \in W}\{\pi(w)\} \ast  \sup_{w \in W}\{\pi(w)\}.$$
	Then we only need to check that the equality holds in complete BL-algebras. We divide the proof in two cases:
	\begin{description}
	\item [$\leq)$]
	Clearly for all $w \in W,$ $\pi(w)  \leq  \sup_{w \in W}\{\pi(w)\}$. Then for every $\in W: $ $$\pi(w) \ast \pi(w)  \leq  \sup_{w \in W}\{\pi(w)\} \ast  \sup_{w \in W}\{\pi(w)\}.$$ Taking supremum on the left side, we obtain the one inequality.
	\item[$\geq)$] For the other, it is easy to see that the equation $x\ast y\to (x^2 \vee y^2)=1$ holds in every BL-chain, thus it holds in every BL-algebra. Then for all $w, \nu \in W$:
	\begin{align*}
		\pi(w) \ast \pi(\nu) &\leq \pi(w)^2 \vee  \pi(\nu)^2\\
		& \leq  \sup_{w \in W}\{\pi(w)^2\} \vee \sup_{\nu \in W}\{\pi(\nu)^2\}\\
		&= \sup_{w \in W}\{\pi(w)^2\}.
	\end{align*}

	Using residuation we have:
	$$\pi(w) \leq \pi(\nu) \to \sup_{w \in W}\{\pi(w)^2\} $$
	Considering the supremum on the left hand side:
	$$\sup_{w \in W} \{ \pi(w)\} \leq \pi(\nu) \to \sup_{w \in W}\{\pi(w)^2\} $$
	and interchanging the left hand side and the antecedent of the right implication, we obtain:
	$$\pi(\nu) \leq \sup_{w \in W} \{ \pi(w)\} \to \sup_{w \in W}\{\pi(w)^2\}. $$
	Taking supremum we get:
	$$\sup_{\nu \in W}\{\pi(\nu)\} \leq \sup_{w \in W} \{ \pi(w)\} \to \sup_{w \in W}\{\pi(w)^2\}. $$
	The proof is completed by using residuation again.
	\end{description}

\end{proof}

Then for a complete BL-algebra ${\bf A}$, as promised before, we have:

\begin{theorem} \label{complex}
	Given a $\Pi{\bf A}$-frame ${\cal P}=\langle W,\pi \rangle$ over a complete BL-algebra $\bf{A}$, its associated complex ${\bf A}$-algebra $\langle {\bf A}^{W},\forall ^{\mathcal{P}}, \exists ^{\mathcal{P}} \rangle$ is a c-EBL-algebra with $c=\pi$.
\end{theorem}

\begin{proof}
	Since ${\bf A}^{W}$ is a BL-algebra, we have to prove that the system $ \langle {\bf A}^{W}, \forall ^{\mathcal{P}}, \exists ^{\mathcal{P}} \rangle$  satisfies axioms E$\forall $ - E5. Let $f$ and $g$ $\in {\bf A}^{W}$:
	\begin{description}
		\item[(E$\forall $)]
		$\forall ^{\mathcal{P}} 1  =  \inf_{w\in W}\{\pi(w) \to 1(w) \} = 1$.
		\item[(E$\exists $)]
		$\exists ^{\mathcal{P}} 0  =  \sup_{w\in W}\{ \pi(w) \ast 0(w)\} = 0$.
		\item[(E1)] To see that $\forall ^{\mathcal{P}}f \to \exists  ^{\mathcal{P}}f =1$ consider the following cases:
		
		-There exists $w' \in W$ such that $\pi(w')=1$, then:\\
		$\forall ^{\mathcal{P}}f =\inf_{w \in W}\{\pi(w) \to f(w)\}\leq \pi(w') \to f(w')= f(w')=1 \ast f(w') = \pi(w') \ast f(w')\leq \sup_{w \in W}\{\pi(w)\ast f(w)\} = \exists ^{\mathcal{P}}f$
		
		- For all $w \in W \ $,  $\pi(w)< 1$. Observe that for every $w \in W$ one has $$\forall ^{\mathcal{P}}(f)(w)  =  \inf_{w\in W}\{\pi(w) \to f(w) \} \leq (\pi \to f)(w).$$ Thus $\forall ^{\mathcal{P}} f \leq \pi \to f.$ This fact together with residuation imply that for every $w \in W$, $$\forall ^{\mathcal{P}}f \ast \pi(w) \leq f(w).$$ By monotonicity of $\ast$, for all $w \in W$ we have $$\forall ^{\mathcal{P}}f \ast \pi(w)^ 2 \leq \pi\ast f (w).$$ Then for every  $w \in W$, $$\forall ^{\mathcal{P}}f \ast \pi(w)^ 2\le \pi\ast f(w) \leq \sup_{w\in W}\{ \pi(w) \ast f(w)\} = \exists ^{\mathcal{P}} f.$$
		Consequently for all $w\in W$,   $\ \pi(w)^2 \leq \forall ^{\mathcal{P}}f \to \exists ^{\mathcal{P}}f$
		and by Lemma \ref{existsc=1}  we get
		$\forall ^{\mathcal{P}}f \leq \exists ^{\mathcal{P}}f$.
		\item[(E2)]
		By definition:
		   \begin{align*}
		\forall ^{\mathcal{P}}(f \to \forall ^{\mathcal{P}} g)&=\inf_{w \in W}\{\pi(w) \to (f(w) \to \forall ^{\mathcal{P}} g(w))\}\\
		&=\inf_{w \in W}\{(\pi(w) \ast f(w)) \to \forall ^{\mathcal{P}} g(w)\}.
			\end{align*}
		Observe that for every $w \in W$, one has $$\exists ^{\mathcal{P}}(f)(w)=\sup_{w \in W}\{\pi(w) \ast f(w)\}\geq (\pi \ast f)(w).$$
		Thus $\pi \ast f \leq \exists ^{\mathcal{P}}f$, and for all $w \in W$:
		$$\exists ^{\mathcal{P}} f \to \forall ^{\mathcal{P}} g\leq (\pi(w) \ast f(w)) \to \forall ^{\mathcal{P}} g(w).$$
		This means that $\exists ^{\mathcal{P}} f \to \forall ^{\mathcal{P}} g$  is a lower bound of the set $\{(\pi(w) \ast f(w)) \to \forall ^{\mathcal{P}} g(w)\}_{w \in W}$.  Therefore 				 $$\exists ^{\mathcal{P}} f \to \forall ^{\mathcal{P}} g \leq \forall ^{\mathcal{P}} (f \to \forall ^{\mathcal{P}} g).$$
	
		Again, by definition, for all $w \in W$ $$\forall ^{\mathcal{P}} (f \to \forall ^{\mathcal{P}} g) \leq (\pi(w)\ast f(w))\to \forall ^{\mathcal{P}} g,$$ and residuation implies  $$\pi(w) \ast f(w) \leq \forall ^{\mathcal{P}}(f \to \forall ^{\mathcal{P}} g) \to \forall ^{\mathcal{P}} g$$  for every $w \in W$. Consequently, $\exists ^{\mathcal{P}} f \leq \forall ^{\mathcal{P}}(f \to \forall ^{\mathcal{P}} g) \to \forall ^{\mathcal{P}} g $, or equivalently,
		$$\forall ^{\mathcal{P}}(f \to \forall ^{\mathcal{P}} g) \leq \exists ^{\mathcal{P}} f \to \forall ^{\mathcal{P}} g.$$
	
		Finally, $\forall ^{\mathcal{P}}(f \to \forall ^{\mathcal{P}} g) = \exists ^{\mathcal{P}} f \to \forall ^{\mathcal{P}} g$.
		\item[(E3)]
		For every $w \in W$,
		\begin{align*}
				\forall ^{\mathcal{P}}(\forall^{\mathcal{P}} f \to g) &=\inf_{w \in W}\{\pi(w) \to (\forall^{\mathcal{P}} f \to g(w))\} \\
				& =\inf_{w \in W}\{\forall^{\mathcal{P}} f \to (\pi(w)\to g(w))\} \\
				& \leq \forall^{\mathcal{P}} f \to (\pi(w)\to g(w)).
			\end{align*}
		Then, for all $w \in W$, $\forall^{\mathcal{P}} f \ast \forall^{\mathcal{P}} (\forall^{\mathcal{P}} f \to g) \leq \pi(w) \to g(w)$.  Hence, $\forall^{\mathcal{P}} f \ast \forall^{\mathcal{P}} (\forall^{\mathcal{P}} f \to g)$ is a lower bound of $\{\pi(w) \to g(w)\}_{w  \in W}$, and therefore $$\forall^{\mathcal{P}} f \ast \forall^{\mathcal{P}} (\forall^{\mathcal{P}} f \to g) \leq \forall^{\mathcal{P}} g$$ or equivalently
	
		$$\forall^{\mathcal{P}} (\forall^{\mathcal{P}} f \to g) \leq \forall ^{\mathcal{P}}f \to \forall ^{\mathcal{P}}g.$$
		
		For the other inequality, as $\forall^{\mathcal{P}} g \leq \pi(w) \to g(w)$, we have
		\begin{align*}
			\forall^{\mathcal{P}} f \to \forall^{\mathcal{P}} g &\leq \forall^{\mathcal{P}} f \to (\pi\to g)(w)\\
			&= \pi(w)\to (\forall^{\mathcal{P}}f \to g)(w)
		\end{align*}
	
		for all $w \in W$. Hence
		$$\forall ^{\mathcal{P}} f \to \forall ^{\mathcal{P}} g \leq \forall ^{\mathcal{P}}(\forall ^{\mathcal{P}} f \to g).$$
				Finally, we have $\forall ^{\mathcal{P}}(\forall ^{\mathcal{P}} f \to g) = \forall ^{\mathcal{P}} f \to \forall ^{\mathcal{P}} g $.
		\item[(E4)]
		For every $w \in W$, $$\exists ^{\mathcal{P}} f \leq \pi(w) \to \exists ^{\mathcal{P}} f.$$ Then $\exists ^{\mathcal{P}} f$ is a lower bound of $\{c(w) \to \exists ^{\mathcal{P}} f\}_{w \in W}$, therefore $$\exists  ^{\mathcal{P}}f \leq \inf_{w \in W}\{\pi(w) \to \exists ^{\mathcal{P}} f\}= \forall ^{\mathcal{P}} \exists  ^{\mathcal{P}}f.$$
		Thus $$\exists  ^{\mathcal{P}}f \to \forall ^{\mathcal{P}} \exists  ^{\mathcal{P}}f = 1.$$
		
		\item[(E4a)] For all $w \in W$,
		\begin{align*}
			\forall ^{\mathcal{P}}(f \wedge g)& =\inf_{w \in W}\{\pi(w) \to (f \wedge g)(w)\}\\
		& \leq \pi(w) \to (f(w) \wedge g(w))\\
		& =(\pi(w) \to f(w)) \wedge (\pi(w) \to g(w)).
		\end{align*}
		
		The right side of the last identity is less than both $\pi(w) \to f(w)$ and $\pi(w) \to g(w)$ for all $w \in W$, then $$\forall ^{\mathcal{P}}(f \wedge g)\leq \forall ^{\mathcal{P}} f \hspace{0.3cm} \mbox{and} \hspace{0.3cm}\forall ^{\mathcal{P}}( f \wedge g)\leq \forall ^{\mathcal{P}} g.$$
		Hence
	
		$$\forall ^{\mathcal{P}}(f \wedge g)\leq \forall ^{\mathcal{P}} f \wedge \forall^{\mathcal{P}} g.$$
	
		On the other hand, for every $w \in W$, $\forall ^{\mathcal{P}} f \wedge \forall^{\mathcal{P}} g \leq \pi(w) \to f(w)$ and similary $\forall ^{\mathcal{P}} f \wedge \forall ^{\mathcal{P}} g \leq \pi(w) \to g(w)$ , therefore $$\forall^{\mathcal{P}}f \wedge \forall^{\mathcal{P}} g \leq (\pi\to f) \wedge (\pi \to g)(w)= \pi\to (f \wedge g)(w)$$ for all $w \in W$. Consequently
	
		$$\forall ^{\mathcal{P}}f \wedge \forall ^{\mathcal{P}} g \leq \forall ^{\mathcal{P}}(f \wedge g).$$
		
		So, $\forall ^{\mathcal{P}}(f \wedge g) = \forall ^{\mathcal{P}}f \wedge \forall ^{\mathcal{P}} g$.
		
		\item[(E4b)] 
		For every $w \in W$, $\pi(w) \ast f(w) \leq \exists ^{\mathcal{P}} f$ and $\pi(w) \ast g(w)\leq \exists ^{\mathcal{P}} g$, thus
	\begin{align*}
				\exists ^{\mathcal{P}} f \vee \exists ^{\mathcal{P}} g & \geq (\pi(w) \ast f(w)) \vee (\pi(w) \ast g(w))\\
				& = \pi(w)\ast (f(w) \vee g(w)).
    \end{align*}
		As $\exists ^{\mathcal{P}}(f \vee g)= \sup_{w \in W}\{\pi(w)\ast (f(w) \vee g(w))\}$, we have
		$$\exists ^{\mathcal{P}}(f \vee g)\leq \exists ^{\mathcal{P}} f \vee \exists ^{\mathcal{P}} g.$$
		
		Moreover, for all $w \in W$, $f(w) \leq f(w) \vee g(w)$, whence $\pi(w) \ast f(w) \leq \pi(w) \ast (f(w) \vee g(w))$. Then $\sup_{w \in W}\{\pi(w) \ast f(w)\} \leq \sup_{w \in W}\{\pi(w) \ast (f(w) \vee g(w))\}$ i.e.,
		 $$\exists ^{\mathcal{P}} f \leq \exists ^{\mathcal{P}} (f \vee g).$$ Analogously, $$\exists ^{\mathcal{P}} g \leq \exists ^{\mathcal{P}} (f \vee g).$$
		Hence, $\exists ^{\mathcal{P}} f \vee \exists ^{\mathcal{P}} g \leq \exists ^{\mathcal{P}} (f \vee g)$.\\
	
		Finally, it holds $\exists ^{\mathcal{P}} (f \vee g)= \exists ^{\mathcal{P}} f \vee \exists ^{\mathcal{P}} g$.
			
		\item[(E5)]
		Since $\pi(w) \ast f(w) \leq \exists ^{\mathcal{P}} f$, for all $w \in W$, one has  $$\pi(w) \ast f(w) \ast \exists ^{\mathcal{P}} g \leq \exists ^{\mathcal{P}} f \ast \exists ^{\mathcal{P}} g.$$
		Since $\exists ^{\mathcal{P}} (f \ast \exists ^{\mathcal{P}} g)=\sup_{w \in W}\{\pi(w) \ast \left(f(w) \ast \exists ^{\mathcal{P}}g(w)\right)\}$, we have $
		\exists ^{\mathcal{P}} (f \ast \exists ^{\mathcal{P}} g) \leq \exists ^{\mathcal{P}} f \ast \exists ^{\mathcal{P}} g$.\\
		
		On the other hand, for all $w \in W$,
		\begin{align*}
				\exists ^{\mathcal{P}}(f \ast \exists ^{\mathcal{P}} g)& =\sup_{w \in W}\{\pi(w) \ast \left(f(w) \ast \exists ^{\mathcal{P}}(w)\right)\}\\
				& \geq\pi(w) \ast \left( f(w) \ast \exists ^{\mathcal{P}} g\right).
		\end{align*}

		Then, by residuation,
		$$\pi(w) \ast f(w)  \leq \exists ^{\mathcal{P}} g \to \exists ^{\mathcal{P}}\left( f \ast \exists ^{\mathcal{P}} g\right).$$
		Therefore $\exists ^{\mathcal{P}} f \leq \exists ^{\mathcal{P}} g \to \exists ^{\mathcal{P}}(f \ast \exists ^{\mathcal{P}} g)$, or equivalently
		
		$$	\exists ^{\mathcal{P}} f \ast \exists ^{\mathcal{P}} g \leq \exists ^{\mathcal{P}} (f \ast \exists ^{\mathcal{P}} g).$$
		Finally, we conclude $\exists ^{\mathcal{P}} (f \ast \exists ^{\mathcal{P}} g) = \exists ^{\mathcal{P}} f \ast \exists ^{\mathcal{P}} g.$
	\end{description}

	By the previous, $\langle{\bf A}^{W}, \forall ^{\mathcal{P}}, \exists ^{\mathcal{P}}\rangle$ is an EBL-algebra, that we will call \emph{complex EBL-algebra}. Moreover, $\pi$ is the focal element of the complex EBL-algebra. Indeed, if $f \in {\bf A}^W$ is such that $\forall^{\mathcal{P}} f=1$, then, by equation (\ref{poss-epis1}), $1=\inf_{w \in W} \{\pi(w) \to f(w)\}$, i.e. $$\pi(w) \leq f(w)$$
	for all $w \in W$. Thus $\pi \leq f$.\\
	 Besides, $\forall^{\mathcal{P}} \pi=\inf_{w \in W} \{\pi(w) \to \pi(w)\}=1$. Therefore, $\pi=\min\{f \in {\bf A}^W: \forall^{\mathcal{P}} f=1\}$ and hence $\langle{\bf A}^{W}, \forall ^{\mathcal{P}}, \exists ^{\mathcal{P}}\rangle$ is a $\pi$-EBL-algebra.

\end{proof}

\medskip

\subsection{Complex algebras defined over BL-chains}

In the attempt to establish a connection between complex c-EBL-algebras and possibilistic frames, we will show that, in the case of the complex c-EBL-algebra given by the Theorem \ref{complex}, if ${\bf A}$ is a BL-chain the subalgebra $\forall^{\mathcal{P}} {\bf A}^W$ coincides with the set of constant maps on $A$, which we will denote $A^*$. Since the images by the operators $\forall^{\mathcal{P}}$ and $\exists^\mathcal{P}$ are constant maps, then $$\forall^\mathcal{P} {\bf A}^W \subseteq A^*.$$  To show the other inclusion, we will need two technical lemmas that strongly depend on the structure of complete MV-chains and of BL-chains. We will try to make the proofs of these lemmas as self-contained as possible, and we refer the readers to  \cite{CDM,BuMon} for details.

\begin{lemma}\label{lemma:mvcompletas} Let ${\bf A}$ be a complete MV-chain. Then ${\bf A}$ is isomorphic to a subalgebra of the standard MV-chain $[0,1]_{\bf MV}.$
\end{lemma}
\begin{proof} Let ${\bf A}$ be an MV-chain whose radical $rad({\bf A})$ (intersection of maximal filters) is non trivial,  i.e., $rad({\bf A})\neq \{1\}$. By Proposition 3.6.4 in \cite{CDM} we have that the infimum of $rad({\bf A})$ does not exist. Therefore if ${\bf A}$ is complete, then it is simple and the result of the lemma follows from Theorem 3.5.1 in the above mentioned book.
\end{proof}

\begin{lemma}\label{Luk God P w, b}
	Let ${\bf A}$ be a complete BL-chain, $W\neq\emptyset$, $a \in A$ and $\pi:W \to A$ is such that $\sup_{w \in W}\pi(w)=1$. Then there are $w' \in W$ and $b \in A$ such that $\pi(w')\to b=a$.
\end{lemma}
\begin{proof}
If ${\bf A}$ is a complete BL-chain, following Corollary 3.2.9 in \cite{BuMon} ${\bf A}$ is isomorphic to an ordinal sum indexed by a totally ordered set $I$ of structures $\{{\bf A}_i: i \in I\}$ each of which is isomorphic to either a complete MV-chain or a cancellative residuated lattice (i.e., a residuated lattice satisfying equation (\ref{eq:producto})). We write ${\bf A}\cong \bigoplus_{i\in I} {\bf A}_i.$

\medskip

Consider $a\in A$ and $i\in I$ such that $a\in A_i.$ If $a=1$ then  take $b=\pi(w')$ for any $w' \in W$. Then we have  $\pi(w')\to b=a$. For the case $a<1$ we have two possible  situations:
\begin{itemize}
\item
there is $j\in I$ and $w'\in W$ such that $j>i$ and $\pi(w')\in A_j.$ In this case, from the definition of $\to$ in the ordinal sum, take $b=a$ and we have that $\pi(w')\to a=a$
\item
$i$ is the maximum element of $I$. This being the case, since $\sup_{w \in W}\pi(w)=1$ there is $W'\subseteq W$ such that $\pi(W')\subseteq A_i$ and $\sup_{w \in W'}\pi(w)=1.$ If ${\bf A}_i$ is an MV-chain, from  Lemma \ref{lemma:mvcompletas} we have that it is isomorphic to a subalgebra of the simple chain $[0,1]_{\bf MV}.$ Take $w' \in W'$ such that $\pi(w') > a$. Then if we consider $b=a + \pi(w') -1$, we have
		\begin{align*}
		\pi(w')\to b&=\min\{1,1 - \pi(w') + b\}\\
		           &=1-\pi(w')+b\\
		           &=1-\pi(w')+a+\pi(w')-1=a.
		\end{align*}
If ${\bf A}_i$ is not an MV-chain, then it is cancellative, i.e., it satisfies equation (\ref{eq:producto}). Then taking $w'\in W'$ and $b=\pi(w')\ast a$ we get $\pi(w') \to b=\pi(w')\to \pi(w')\ast a=a$. \end{itemize}
\end{proof}

\begin{theorem}
	For each $a \in A$, the constant map $f_a \in A^W$, given by $f_a(w)=a$  for all $w \in W$  belongs to  $\forall^{\mathcal{P}}{\bf A}^W$, i.e., $\forall^{\mathcal{P}}{\bf A}^W=A^*$, where $A^*$ is the set of all constant maps in $A^W$.
\end{theorem}
\begin{proof}
	Let $a \in A$ and $f_a \in A^W,$ given by $f_a(w)=a$ for all $w \in W$. By the previous Lemma, exist $w' \in W$ and $b \in A$ such that $\pi(w')\to b=a$. Take the function $g \in A^W$ given by $$g(w)=\begin{cases} b & \mbox{si } w=w'\\
	1 & \mbox{si } w\neq w'.
	\end{cases}$$
	Therefore $\forall^{\mathcal{P}}g=\inf_{w \in W}\{\pi(w)\to g(w) \}=f_a$ and $f_a \in \forall^{\mathcal{P}}{\bf A}^W$ as desired.
\end{proof}

\begin{theorem}\label{possibilistic frame associated}
	Let ${\bf A}$ be a complete BL-chain, $W \neq \emptyset$ and $A^W$ the set of functions from $W$ into $A$. If $\mathcal{A}=\langle A^W, \forall, \exists \rangle$ is a c-EBL-algebra such that $\sup_{w \in W}c(w)=1$ and $\forall A^W=A^*$,  then $\mathcal{P}^{\mathcal{A}}=\langle W, c \rangle $ is a possibilistic frame. Moreover, the complex c-EBL-algebra associated to $\mathcal{P}^{\mathcal{A}}$ satisfies $$\forall^{\mathcal{P}^{\mathcal{A}}} = \forall \hspace{2cm} \exists^{\mathcal{P}^{\mathcal{A}}} = \exists. $$
\end{theorem}
\begin{proof}
	Observe that since $W \neq \emptyset$ and $\sup_{w \in W}c(w)=1$, then $\mathcal{P}^{\mathcal{A}}=\langle W, c \rangle $ is a possibilistic frame. In this sense, let  $\langle{\bf A}^{W}, \forall ^{\mathcal{P}^{\mathcal{A}}}, \exists ^{\mathcal{P}^{\mathcal{A}}}\rangle$ be the complex EBL-algebra associated with $\mathcal{P}^{\mathcal{A}}=\langle W,c \rangle$ given by the Theorem \ref{complex}, that is, for every $f \in A^W$:
	$$\forall^{\mathcal{P}^{\mathcal{A}}} f(w)=\inf_{w \in W}\{c(w) \to f(w)\},$$
	$$\exists^{\mathcal{P}^{\mathcal{A}}} f(w)=\sup_{w \in W}\{c(w) \ast f(w)\}.$$
	
	We will show that these operators coincide with $\forall$ and $\exists$, i.e., for every $f \in A^W$ and for every $w' \in W$: $\forall^{\mathcal{P}^{\mathcal{A}}} f(w') = \forall f(w')$ and $\exists^{\mathcal{P}^{\mathcal{A}}} f(w')=\exists f(w')$.\\
	Let $f \in A^W$. From Theorem \ref{c-teorema}, we know that
	\begin{equation}\label{eq forall }
	\forall f=\max\{g \in \forall A^W: g \leq c \to f\},
	\end{equation}
	\begin{equation}\label{eq exists}
	\exists f=\min\{g \in \forall A^W: c \ast f \leq g \}.
	\end{equation}
	Note that since $\forall A^W = A^*$, then $\forall f$ and $\exists f$ are constants maps, so we will write $\forall f$ instead of $\forall f (w)$ for any $w \in W$.\\
	First, by (\ref{eq forall }), $\forall f \leq c \to f$, that is, $\forall f \leq c(w) \to f(w)$, for all $w \in W$. Consequently, $\forall f$ is a lower bound of $\{c(w) \to f(w), w \in W\}$ and therefore, $\forall f \leq \forall^{\mathcal{P}^{\mathcal{A}}} f$.\\
	On the other hand, by definition of $\forall^{\mathcal{P}^\mathcal{A}}$ we have
	 $$\forall^{\mathcal{P}^\mathcal{A}}(f) \leq c(w)\to f(w)$$ for every $w \in W$
	, i.e., $\forall^{\mathcal{P}^\mathcal{A}} f \leq c \to f$. Now, since $\forall^{\mathcal{P}^\mathcal{A}}$ is a constant map, $\forall^{\mathcal{P}^\mathcal{A}} \in \forall A^W$. Hence, by (\ref{eq forall }), we conclude $\forall^{\mathcal{P}^\mathcal{A}} f \leq \forall f$.	In analogous way, we have the result for $\exists^{\mathcal{P}^{\mathcal{A}}} f=\exists f$ using (\ref{eq exists}).
\end{proof}

\begin{remark} It is worth to observe that the hypotheses of the Theorem \ref{possibilistic frame associated} are necessary. In details, we have:
  \begin{itemize}
  	\item There are c-EBL-algebras of the form ${\bf A}^X$ with $X\neq \emptyset$ such that $\sup_{x \in X}c (x) \neq 1$. For example,  consider ${\bf A}$ as the finite MV-chain $\mbox{\L}_4$ and $X=\mathbb{N}$. We define for every $f \in A^X$, $$\forall f(n)=\begin{cases}
  	1&\mbox{if } f(n)\geq \frac{2}{3}\mbox{ for all } n\in\mathbb{N},\\
  	0&\mbox{otherwise.} \end{cases}$$ $$ \exists f(n)=\begin{cases}
1&\mbox{if } f(n)\geq \frac{2}{3}\mbox{ for some } n\in\mathbb{N},\\
0&\mbox{otherwise.} \end{cases}$$
 The reader can easily corroborate that the resulting structure $\langle A^\mathbb{N}, \forall, \exists \rangle$ is a c-EBL-algebra with focal element $c:\mathbb{N} \to A$ such that $c(n)=\frac{2}{3}$ for every $n \in \mathbb{N}$. Clearly,  $\sup_{n \in \mathbb{N}}c(n) \neq 1$. Note also that $\forall A^\mathbb{N} =\{0,1\}$, i.e., $\forall A^\mathbb{N} \subsetneq A^*$. \smallskip
  	\item There exist c-EBL-algebras  of the form $A^X$ where $\sup_{x \in X} c(x)=1$, but $\forall A^X \neq A^*$.  For this case, first consider the epistemic estructure ${\bf A}$ over $\mbox{\L}_4$  given by  $\forall a=1$ if $a=1$ and $\forall a=0$  if not, while $\exists a=0$ if $a=0$ and $\exists a=1$ otherwise.
   Consider the product algebra ${\bf A}^{\mathbb{N}}$, i.e., for each $f \in A^\mathbb{N}$ $$(\bar{\forall} f)(n)= \forall(f(n))  \hspace{0.5cm}(\bar{\exists} f)(n)= \exists(f(n)).$$
  Clearly, since ${\bf A}$ is a c-EBL-algebra, ($c=1$) then  $\langle A^\mathbb{N}, \bar{\forall}, \bar{\exists} \rangle$ is also a $\bar{c}$-EBL-algebra with $\bar{c}=c(n)$. We have that $\sup_{n \in \mathbb{N}}c(n) = 1$, however the quantifiers are not constant maps, therefore $\forall A^X \neq A^*$.
  \end{itemize}

\end{remark}

\section{Conclusion and future work}

The motivation of our paper is to present an algebraic characterization of a system of fuzzy epistemic logic that extends the classical $KD45$ and it is based on H\'{a}jek's fuzzy system basic logic.
To achieve our aim, we have introduced Epistemic BL-algebras, as BL-algebras with two unary operators that behave generalizing the modal operators of $KD45$. We have studied some of their logical and algebraic properties and we have shown their relationship with Monadic BL-algebras, which turn to be the algebraic counterpart of the fuzzy version of $S5$. The results of Section \ref{sec:2} suggest that our definition of EBL-algebras is on firm ground: we have shown that EBL-algebras whose BL-reduct are Boolean algebras coincide with the algebraic correspondent of classical $KD45$ (Pseudomonadic Algebras) and that the ones whose BL-reduct is a G\"{o}del algebra are equivalent to serial transitive and euclidean bi-modal G\"{o}del algebras, the algebraic correspondent to the G\"{o}del generalization of $KD45$.
To close the ideas of the paper, after investigating c-EBL-algebras and complex EBL-algebras, we have proved that the algebras of functions in the fuzzy possibilistic Kripke frames defined by \cite{HajekBook98} are EBL-algebras.

Though the reported results go in the directions of our goal there is still a lot to investigate. We plan to continue our research focusing on the study of the subvariety of Epistemic MV-algebras, that is, EBL-algebras whose BL-reduct is an MV-algebra. As far as we know there are no previous result in this direction. Our goal will then be to establish a connection between this new EMV-algebras and the modal logic given by a finite MV-chain studied by \cite{BoEsGoRo11}.

There is also a problem that needs to be solved in the future and it is to establish if both, our algebraic semantics and the Kripke-style semantics presented by \cite{HajekBook98} correspond to the same axiomatic system.

We think that EBL-algebras pose many new challenges that have to be faced.

\section*{Compliance with Ethical Standards}
Funding: The results of this paper are framed
in the following research project: PIP 112-20150100412CO, CONICET, \textit{Desarrollo de Herramientas Algebraicas y Topol\'ogicas para el Estudio de L\'ogicas de la Incertidumbre y la Vaguedad. DHATELIV.} Penelope Cordero was supported by a CONICET grant during the preparation of the paper.

Conflict of Interest: All authors declare that they have
no conflict of interest.

Ethical approval: This article does not contain any
studies with human participants or animals performed
by any of the authors.

\end{document}